\documentclass{amsart}
\usepackage{amsmath,amssymb}
\usepackage{amsthm,thmtools}
\usepackage{mathtools}
\usepackage{array}
\usepackage[shortlabels]{enumitem}
\usepackage{multicol}
\usepackage{graphicx} 
\usepackage{xcolor}
\usepackage{ytableau}
\usepackage[margin=1.25in]{geometry}
\usepackage{microtype}
\usepackage{cleveref}

\definecolor{customred}{RGB}{180,50,50}

\newtheorem{theorem}{Theorem}[section]

\newtheorem{corollary}[theorem]{Corollary}
\newtheorem{lemma}[theorem]{Lemma}
\newtheorem{problem}[theorem]{Problem}
\newtheorem{proposition}[theorem]{Proposition}

\theoremstyle{definition}
\newtheorem{definition}[theorem]{Definition}

\newtheorem{remark}[theorem]{Remark}
\newtheorem{example}[theorem]{Example}

\usepackage{tikz}
\usepackage{tikz-cd}
\usepackage{tkz-graph}
\usetikzlibrary{arrows.meta,tikzmark,decorations.markings,fit,calc,shapes,matrix,math}
\usetikzlibrary{calc, shapes, backgrounds,arrows,positioning,plotmarks,decorations.pathreplacing}
\tikzset{>=stealth',
  head/.style = {fill = white, text=black},
  plaque/.style = {draw, rectangle, minimum size = 10mm, fill=white}, 
     pil/.style={->,thick},
  junct/.style = {draw,circle,inner sep=0.5pt,outer sep=0pt, fill=black}
  }
\tikzcdset{scale cd/.style={every label/.append style={scale=#1},
    cells={nodes={scale=#1}}}}
\tikzset{
  otimes/.style={
    draw=none,
    every to/.append style={
      edge node={node [sloped, allow upside down, auto=false]{$\otimes$}}}
  }
}

\DeclareMathOperator{\RS}{RS}
\DeclareMathOperator{\SYT}{SYT}
\DeclareMathOperator{\sort}{sort}
\DeclareMathOperator{\stack}{stack}
\DeclareMathOperator{\stacked}{stacked}
\DeclareMathOperator{\NE}{NE}
\DeclareMathOperator{\NW}{NW}
\DeclareMathOperator{\SE}{SE}
\DeclareMathOperator{\SW}{SW}

\DeclareMathOperator{\prom}{\mathsf{prom}}
\DeclareMathOperator{\gromotion}{\mathcal{G}}
\DeclareMathOperator{\promotion}{\mathcal{P}}
\DeclareMathOperator{\evacuation}{\mathcal{E}}
\DeclareMathOperator{\PEdiagram}{\mathcal{P}\mathcal{E}-\mathsf{diagram}}
\DeclareMathOperator{\PM}{\mathbf{M}}

\usepackage{todonotes}

\newcommand{\vsum}{+'}

\newcommand{\too}[1]{\stackrel{#1}{\to}}

\title{Promotion permutations and the \\ Robinson--Schensted correspondence}
\author{Stephan Pfannerer and Joshua P. Swanson}
\date{September 2025}

\begin{document}

\begin{abstract}
    \emph{Promotion permutations} have recently been associated to each rectangular standard Young tableau \cite{fluctuating-paper}. Here we relate promotion permutations to the Robinson--Schensted (RS) correspondence. More precisely, we show that taking a pair of standard Young tableaux of the same rectangular shape, \emph{stacking} them, and computing the middle promotion permutation yields the RS permutation of the pair up to simple twists. Moreover, the full list of promotion permutations in this special case encodes Viennot's \emph{geometric shadow line} construction. As a consequence, we characterize a subset of the collection of possible promotion permutations in terms of crossing and nesting numbers.
\end{abstract}

\maketitle

\section{Introduction}\label{sec:intro}

Let $\SYT(r \times c)$ denote the set of standard Young tableaux of rectangular shape with $r$ rows, $c$ columns, and $n \coloneqq rc$ cells. The well-known combinatorial bijections of \emph{promotion} $\promotion$ and \emph{evacuation} $\evacuation$ are particularly well-behaved on $\SYT(r \times c)$ \cite{Schutzenberger-promotion,Haiman,Stanley-promotion-evacuation,fluctuating-paper}, e.g.,~they satisfy the dihedral relations $\promotion^n = \evacuation^2 = 1$ and $\evacuation \circ \promotion \circ \evacuation = \promotion^{-1}$. There has been significant interest in providing a direct pictorial understanding of these symmetries, which are not at all obvious on the level of tableaux.

One such explanation was recently provided by \cite{fluctuating-paper} and is as follows. Together with Gaetz, Pechenik, and Striker, we introduced the \emph{promotion permutations}
\[\prom_\bullet(T) = (\prom_1(T), \ldots, \prom_{r-1}(T))
\]
of $T \in \SYT(r \times c)$, where $\prom_i(T) \in \mathfrak{S}_n$ for $1 \leq i \leq r-1$; see \Cref{sec:background-prom-perms} for a quick summary and \Cref{sec:background-Fomin} for further context. The map
\begin{equation}
  \SYT(r \times c) \to \mathfrak{S}_n^{r-1}\qquad\text{ where }\qquad T \mapsto \prom_\bullet(T)
\end{equation}
is injective \cite[Thm.~5.9]{fluctuating-paper}. Drawing $\prom_\bullet(T)$ in a disc, $\promotion$ and $\evacuation$ correspond precisely to rotation and reflection, respectively \cite[\S7]{fluctuating-paper}, making the dihedral relations obvious; see \Cref{sec:background-prom-perms}.

Separately, the \emph{Robinson--Schensted} (RS) correspondence is a very well known bijection
\begin{equation}
  \RS \colon \mathfrak{S}_n \too{\sim} \bigsqcup_{\lambda \vdash n} \SYT(\lambda) \times \SYT(\lambda).
\end{equation}

In the special case when $\lambda = r \times c$, one may ask whether there is a relationship between $\RS$ and $\prom_\bullet$.
%
%
Given a pair $P, Q \in \SYT(r \times c)$, we may \emph{stack} them as in \Cref{fig:stacking-example}, resulting in $S \coloneqq \stack(P, Q) \in \SYT(2r \times c)$. We write $\prom_i^{\NE}(S)$ for the list defined by $\prom_i^{\NE}(S)(j) = \prom_i(S)(j)-n$ for $1 \leq j \leq n$. We show:

\begin{theorem}\label{thm:prom-RS}
    If $P, Q$ are rectangular standard Young tableaux with $r$ rows, then
    \begin{equation}\label{eq:prom-RS}
        \prom_r^{\NE}(\stack(P, Q)) = \RS^{-1}(Q^\top, P^\top).
    \end{equation}
\end{theorem}

\begin{figure}[hbtp]
    \[
        \begin{array}{cc}
        \begin{array}{c}
        P = \vcenter{\hbox{\ytableaushort{1356,279{11},48{10}{12}}\vspace{1em}}} \\[6pt]
        Q = \vcenter{\hbox{\ytableaushort{1237,456{10},89{11}{12}}}}
        \end{array}
        &
        \stack(P, Q) = \vcenter{\hbox{\ytableaushort{1356,279{11},48{10}{12},{13}{14}{15}{19},{16}{17}{18}{22},{20}{21}{23}{24}}}}
        \end{array}
    \]
    \caption{Example of stacking tableaux. Here we add $3 \times 4 = 12$ to the entries in $Q$ and place them below $P$.}
    \label{fig:stacking-example}
\end{figure}

\begin{example}\label{ex:prom-RS}
    Let $P, Q$ be as in \Cref{fig:stacking-example}, so $r=3, c=4, n=12$. One finds that
    \begin{align*}
      \prom_3(\stack(P, Q))
        &= [19, 22, 15, 24, 14, 13, 18, 23, 17, 21, 16, 20, 6, 5, 3, 11, 9, 7, 1, 12, 10, 2, 8, 4] \in \mathfrak{S}_{24} \\
      \prom_3^{\NE}(\stack(P, Q))
        &= [7, 10, 3, 12, 2, 1, 6, 11, 5, 9, 4, 8] \in \mathfrak{S}_{12}.
    \end{align*}
    Applying the usual $\RS$ insertion algorithm to the latter permutation results in $(Q^\top, P^\top)$. 
\end{example}

Characterizing the image of $\prom_\bullet$ is an open problem; see \Cref{prob:prom-perms}. However, \Cref{thm:main} gives a partial answer; see \Cref{fig:matching-alt} for an example. Let $\SYT^{\stacked}(2r \times c)$ denote the subset of $\SYT(2r \times c)$ consisting of tableaux of the form $\stack(P, Q)$ for $P, Q \in \SYT(r \times c)$.

\begin{corollary}\label{cor:prom-RS}
    Let $n = rc$. The map $\SYT^{\stacked}(2r \times c) \to \mathfrak{S}_{2n}$ given by $S \mapsto \prom_r(S)$ is injective. The image is precisely the set of fixed-point free involutions $\rho$ on $\mathfrak{S}_{2n}$ where each $2$-cycle $(a\ b)$ of $\rho$ is of the form $a \in \{1, \ldots, n\}, b \in \{n+1, \ldots, 2n\}$, the crossing number of $\rho$ is $r$, and the nesting number of $\rho$ is $c$.
\end{corollary}

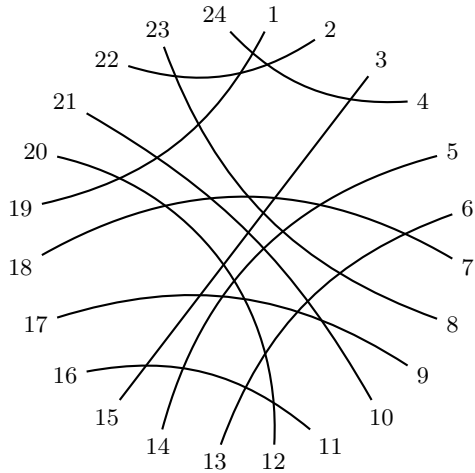
\begin{figure}[hbtp]
    \centering
    
    \begin{tikzpicture}
      \def\r{3cm}
    
      \foreach \i in {1,...,24} {
        \node[font=\small] (v\i) at ({90-7.5-(\i-1)*15}:\r) {\i};
      }
    
      \draw[thick, bend left=25] (v1) to (v19);
      \draw[thick, bend left=25] (v2) to (v22);
      \draw[thick] (v3) to (v15);
      \draw[thick, bend left=25] (v4) to (v24);
      \draw[thick, bend right=30] (v5) to (v14);
      \draw[thick, bend right=25] (v6) to (v13);
      \draw[thick, bend right=30] (v7) to (v18);
      \draw[thick, bend left=25] (v8) to (v23);
      \draw[thick, bend right=25] (v9) to (v17);
      \draw[thick, bend right=15] (v10) to (v21);
      \draw[thick, bend right=25] (v11) to (v16);
      \draw[thick, bend right=40] (v12) to (v20);
    \end{tikzpicture}

    \caption{The perfect matching of $\prom_3(\stack(P, Q))$ for $P, Q \in \SYT(3 \times 4)$ from \Cref{ex:prom-RS}. Observe that the maximum number of mutually crossing strands is $3$ (e.g., $\{1, 19\}, \{2, 22\}, \{4, 24\}$) and the maximum number of nested strands is $4$ (e.g., $\{1, 19\}, \{3, 15\}, \{5, 14\}, \{6, 13\})$, in agreement with \Cref{cor:prom-RS}.}
    \label{fig:matching-alt}
\end{figure}

\begin{remark}\label{rem:this}
    Indeed, in the rectangular case, we may recover the entirety of Viennot's shadow line construction, including all ``skeleta,'' by examining the other promotion permutations in the northeast (or southwest) quadrant. See \Cref{thm:main} and \Cref{ex:main} for the corresponding generalization of \eqref{eq:prom-RS}.
\end{remark}

The rest of the paper is organized as follows. In \Cref{sec:background}, we give background on promotion and evacuation; promotion permutations; local rules, growth diagrams, and decorations; Fomin's local rules; crossing and nesting numbers; and Viennot's geometric shadow line construction. In \Cref{sec:main}, we state and prove our main result, \Cref{thm:main}, which implies \Cref{thm:prom-RS}, and we deduce \Cref{cor:prom-RS}.

\subsection{Acknowledgments}

The authors would like to thank Ron Cherny, Christian Gaetz, Oliver Pechenik, and Jessica Striker for enlightening discussions. Pfannerer were partially supported by Olya Mandelshtam's and Oliver Pechenik's Discovery Grants (RGPIN-2021-02568, RGPIN-2021-02391) and Pechenik's Launch Supplement (DGECR-2021-00010) from the Natural Sciences and Engineering Research Council of Canada. Swanson was partially supported by NSF grant DMS-2348843. This work was completed while Swanson was in residence at ICERM, which is supported by NSF grant DMS-1929284, during the Fall 2025 Semester Program on Categorification and Computation in Algebraic Combinatorics.

\section{Background}\label{sec:background}
We use standard notation for partitions, tableaux, and the RS correspondence as in \cite[Ch.~7]{Stanley:EC2} or \cite[Ch.~3]{Sagan-symmetric-group}. 

\subsection{Promotion and evacuation}\label{sec:background-promotion-evacuation}

We use three related combinatorial operations  on standard Young tableaux associated with Sch\"utzenberger \cite{Schutzenberger-promotion}. They are based on Sch\"utzenberger's \emph{jeu de taquin} (see \cite[A1.2]{Stanley:EC2} or \cite[\S3.7]{Sagan-symmetric-group}). Let $T \in \SYT(\lambda)$ where $\lambda \vdash n$.
\begin{itemize}
    \item \emph{Promotion} deletes $1$ from $T$, rectifies by applying jeu de taquin slides until the empty cell hits an outer corner of $\lambda$, fills this outer corner with $n+1$, and decrements all entries by $1$. The result is $\promotion(T) \in \SYT(\lambda)$.
    \item \emph{Evacuation} successively deletes $1, 2, \ldots, n$ from $T$, rectifying each time, and creates $\evacuation(T) \in \SYT(\lambda)$ by reading the sequence of outer corners in reverse.
    \item \emph{Gromotion}\footnote{This term was coined in \cite{promotion-digraphs}, where instead the finite alphabet $[n]$ itself rotates from $1 < \cdots < n$ to $2 < \cdots < n < 1$ to $3 < \cdots < n < 1 < 2$, etc. One may consider our version as taking place on the infinite alphabet $\mathbb{Z}$. The construction in \cite{promotion-digraphs} generalizes \cite{fluctuating-paper} to produce \emph{promotion digraphs} associated to increasing or non-rectangular tableaux. All constructions agree for $\SYT(r \times c)$. Whether to use gromotion on an infinite or finite alphabet to compute promotion permutations is a stylistic preference and not a substantive difference.} (on an infinite alphabet) is the same as promotion, except the entries are not decremented. Instead, under repeated applications of gromotion, the labels change from $1 < \cdots < n$ to $2 < \cdots < n+1$ to $3 < \cdots < n+2$, etc. Write $\gromotion(T)$ for the gromotion of $T$.
\end{itemize}

\begin{example}\label{ex:PEG}
    We apply the various operations to $Q$ from \Cref{fig:stacking-example}. For later use, we indicate the sliding path in powers of gromotion, with upward slides into the first or second row highlighted in orange or teal, respectively.

    \newsavebox{\exPEGa}
    \sbox{\exPEGa}{
        \begin{tikzpicture}[baseline={([yshift=-.6ex]current bounding box.center)}]
        \matrix (m) [matrix of math nodes,
                     nodes={draw, minimum size=6.2mm, anchor=center},
                     column sep=-\pgflinewidth,
                     row sep=-\pgflinewidth,
                     ampersand replacement=\&
                     ]
        {
         1  \&  2 \&  3 \&  7\\
         4  \&  5 \&  6 \& 10\\
         8  \&  9 \& 11 \& 12\\
        };
        \draw[->,shorten >=0.15cm,shorten <=0.15cm,black,thick] (m-1-2.center) -- (m-1-1.center);
        \draw[->,shorten >=0.15cm,shorten <=0.15cm,black,thick] (m-1-3.center) -- (m-1-2.center);
        \draw[->,shorten >=0.15cm,shorten <=0.15cm,orange,thick] (m-2-3.center) -- (m-1-3.center);
        \draw[->,shorten >=0.15cm,shorten <=0.15cm,black,thick] (m-2-4.center) -- (m-2-3.center);
        \draw[->,shorten >=0.15cm,shorten <=0.15cm,teal,thick] (m-3-4.center) -- (m-2-4.center);
    \end{tikzpicture}}
    
    \newsavebox{\exPEGb}
    \sbox{\exPEGb}{
        \begin{tikzpicture}[baseline={([yshift=-.6ex]current bounding box.center)}]
        \matrix (m) [matrix of math nodes,
                     nodes={draw, minimum size=6.2mm, anchor=center},
                     column sep=-\pgflinewidth,
                     row sep=-\pgflinewidth,
                     ampersand replacement=\&
                     ]
        {
         2  \&  3 \&  6 \&  7\\
         4  \&  5 \& 10 \& 12\\
         8  \&  9 \& 11 \& 13\\
        };
        \draw[->,shorten >=0.15cm,shorten <=0.15cm,black,thick] (m-1-2.center) -- (m-1-1.center);
        \draw[->,shorten >=0.15cm,shorten <=0.15cm,orange,thick] (m-2-2.center) -- (m-1-2.center);
        \draw[->,shorten >=0.15cm,shorten <=0.15cm,teal,thick] (m-3-2.center) -- (m-2-2.center);
        \draw[->,shorten >=0.15cm,shorten <=0.15cm,black,thick] (m-3-3.center) -- (m-3-2.center);
        \draw[->,shorten >=0.15cm,shorten <=0.15cm,black,thick] (m-3-4.center) -- (m-3-3.center);
        \end{tikzpicture}
    }
    
    \newsavebox{\exPEGc}
    \sbox{\exPEGc}{
        \begin{tikzpicture}[baseline={([yshift=-.6ex]current bounding box.center)}]
        \matrix (m) [matrix of math nodes,
                     nodes={draw, minimum size=6.2mm, anchor=center},
                     column sep=-\pgflinewidth,
                     row sep=-\pgflinewidth,
                     ampersand replacement=\&
                     ]
        {
         3  \&  5 \&  6 \&  7\\
         4  \&  9 \& 10 \& 12\\
         8  \& 11 \& 13 \& 14\\
        };
        \draw[->,shorten >=0.15cm,shorten <=0.15cm,orange,thick] (m-2-1.center) -- (m-1-1.center);
        \draw[->,shorten >=0.15cm,shorten <=0.15cm,teal,thick] (m-3-1.center) -- (m-2-1.center);
        \draw[->,shorten >=0.15cm,shorten <=0.15cm,black,thick] (m-3-2.center) -- (m-3-1.center);
        \draw[->,shorten >=0.15cm,shorten <=0.15cm,black,thick] (m-3-3.center) -- (m-3-2.center);
        \draw[->,shorten >=0.15cm,shorten <=0.15cm,black,thick] (m-3-4.center) -- (m-3-3.center);
        \end{tikzpicture}
    }
    
    \newsavebox{\exPEGd}
    \sbox{\exPEGd}{
        \begin{tikzpicture}[baseline={([yshift=-.6ex]current bounding box.center)}]
        \matrix (m) [matrix of math nodes,
                     nodes={draw, minimum size=6.2mm, anchor=center},
                     column sep=-\pgflinewidth,
                     row sep=-\pgflinewidth,
                     ampersand replacement=\&
                     ]
        {
         1  \&  2 \&  4 \&  5\\
         3  \&  7 \&  8 \&  9\\
         6  \& 10 \& 11 \& 12\\
        };
        \end{tikzpicture}
    }
    
    \newsavebox{\exPEGe}
    \sbox{\exPEGe}{
        \begin{tikzpicture}[baseline={([yshift=-.6ex]current bounding box.center)}]
        \matrix (m) [matrix of math nodes,
                     nodes={draw, minimum size=6.2mm, anchor=center},
                     column sep=-\pgflinewidth,
                     row sep=-\pgflinewidth,
                     ampersand replacement=\&
                     ]
        {
         1  \&  2 \&  5 \&  6\\
         3  \&  4 \&  9 \& 11\\
         7  \&  8 \& 10 \& 12\\
        };
        \end{tikzpicture}
    }
    
    \newsavebox{\exPEGf}
    \sbox{\exPEGf}{
        \begin{tikzpicture}[baseline={([yshift=-.6ex]current bounding box.center)}]
        \matrix (m) [matrix of math nodes,
                     nodes={draw, minimum size=6.2mm, anchor=center},
                     column sep=-\pgflinewidth,
                     row sep=-\pgflinewidth,
                     ampersand replacement=\&
                     ]
        {
         1  \&  3 \&  4 \&  5\\
         2  \&  7 \&  8 \& 10\\
         6  \&  9 \& 11 \& 12\\
        };
        \end{tikzpicture}
    }
    
    \begin{center}
    \begin{tikzcd}
        {\usebox{\exPEGa}}
          \rar{\gromotion} \ar{dr}{\promotion} \ar[d,leftrightarrow,"\evacuation"]
          & {\usebox{\exPEGb}}
          \rar{\gromotion} \dar{-1}
          & {\usebox{\exPEGc}}
          \rar{\gromotion} \dar{-2}
          & \cdots \\
        {\usebox{\exPEGd}}
          & {\usebox{\exPEGe}}
          \rar{\promotion}
          & {\usebox{\exPEGf}}
          \rar{\promotion}
          & \cdots
    \end{tikzcd}
    \end{center}
\end{example}

\subsection{Promotion permutations}\label{sec:background-prom-perms}
Promotion permutations were introduced in \cite{fluctuating-paper}, building on work of Hopkins--Rubey \cite{Hopkins-Rubey}. Six equivalent definitions were given in \cite{fluctuating-paper} in the generality of \emph{fluctuating tableaux}. Here we summarize just two of these in the special case of rectangular standard tableaux.

\begin{definition}\label{def:prom-perms}
    Let $T \in \SYT(r \times c)$. The $i$-th \emph{promotion permutation} of $T$ lists the entries which slide into row $i$ when computing $\gromotion(T), \gromotion^2(T), \ldots, \gromotion^n(T)$, modulo $n \coloneqq rc$ (with residues in $\{1, \ldots, n\}$).
\end{definition}

\begin{example}\label{ex:prom-sliding}
    For $Q$ from \Cref{fig:stacking-example}, we see $\prom_1(Q) = [6, 5, 4, \ldots]$ since in \Cref{ex:PEG}, $6$ slides into the first row when computing $\gromotion(Q)$, then $5$ slides into the first row when computing $\gromotion^2(Q)$, etc. Similarly we see $\prom_2(Q) = [12, 9, 8, \ldots]$ since these entries slide into the second row. In all,
    \begin{align*}
        \prom_1(Q) &= [6, 5, 4, 12, 10, 9, 8, 3, 2, 11, 7, 1] \\
        \prom_2(Q) &= [12, 9, 8, 3, 2, 1, 11, 7, 6, 5, 10, 4].
    \end{align*}
\end{example}

The astute reader may observe in \Cref{ex:prom-sliding} that $\prom_2(Q) = \prom_1(Q)^{-1}$. This is part of a general phenomenon and is broadly related to the dihedral symmetry of $\promotion$ and $\evacuation$ acting on $\SYT(r \times c)$. Let $\sigma = (1\,2\,\cdots\,n)$ (in cycle notation) and $w_0 = [n-1, \ldots, 1]$ (in one-line notation). We have the following.

\begin{theorem}[{\cite[Thm.~6.7]{fluctuating-paper}}\protect\footnote{A typo in the original published version of \cite[Thm.~6.7]{fluctuating-paper} used $i$ instead of $r-i$ on the right-hand side of the equivalent of what we have called \Cref{thm:prom-perm-symmetries}(c).}]\label{thm:prom-perm-symmetries}
    Let $T \in \SYT(r \times c)$ and $1 \leq i \leq r-1$. Then
    \begin{enumerate}[(a)]
        \item $\prom_i(T)$ is a permutation of $\{1, \ldots, n\}$ where $n = rc$,
        \item $\prom_i(\promotion(T)) = \sigma^{-1} \prom_i(T) \sigma$,
        \item $\prom_i(\evacuation(T)) = w_0 \prom_{r-i}(T) w_0$, and
        \item $\prom_{r-i}(T) = \prom_i(T)^{-1}$.
    \end{enumerate}
\end{theorem}

\begin{remark}
    \Cref{thm:prom-perm-symmetries}(b)-(c) show that promotion corresponds to rotation clockwise by $2\pi/(rc)$ and that evacuation corresponds to reflection through the $y$-axis together with tuple reversal; see \cite[\S7]{fluctuating-paper} and \Cref{ex:prom-dihedral} for details. When $T \in \SYT(2r \times c)$, then \Cref{thm:prom-perm-symmetries}(d) says that $\prom_r(T)$ is an involution. It is clear from \Cref{def:prom-perms} that promotion permutations are fixed-point free.
\end{remark}

\begin{example}\label{ex:prom-dihedral}
    By computing powers of gromotion, we find
    \begin{center}
    \begin{tabular}{m{2cm}m{7cm}m{5cm}}
    $T = \vcenter{\hbox{\ytableaushort{13,25,46}}}$
    &
    $\begin{array}{l}
      \prom_1(T) = [2, 5, 4, 1, 6, 3] = (1\,2\,5\,6\,3\,4) \\
      \prom_2(T) = [4, 1, 6, 3, 2, 5] = \prom_1(T)^{-1}
    \end{array}$
    &
    \includegraphics[scale=0.35, trim={0.4cm 3.5cm 11.0cm 0.35cm},clip]{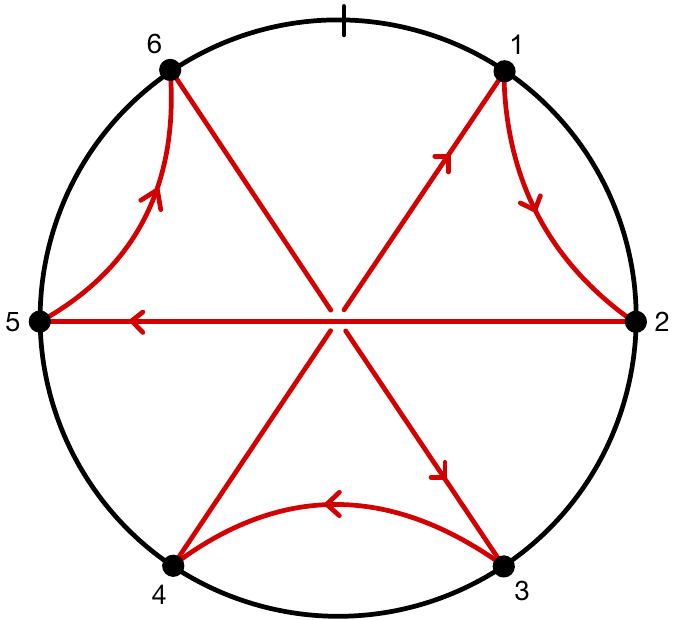}
    \end{tabular}
    \end{center}
    The diagram depicts $\prom_1(T)$ as a directed graph. By \Cref{thm:prom-perm-symmetries}, for $r=3$ we have $\prom_2(T) = \prom_1(T)^{-1}$, so $\prom_2$ may generally be obtained in this case by reversing the arrows. One may compute $\promotion(T) = 12/34/56$ and $\prom_1(\promotion(T)) = [4, 3, 6, 5, 2, 1] = (2\,3\,6\,1\,4\,5)$, so indeed $\promotion$ rotates the digraph clockwise by $2\pi/6$. Since the digraph exhibits a $3$-fold symmetry, $\promotion^2(T) = T$ here. Similarly, evacuation reflects the digraph through the vertical axis and swaps $\prom_1$ with $\prom_2$. Hence we see this $T$ is self-evacuating, i.e., $\evacuation(T) = T$, as may be verified directly. For $r \geq 4$, one genuinely must track multiple promotion permutations simultaneously.
\end{example}

\subsection{Local rules, growth diagrams, and decorations}\label{sec:background-local-rules}

While \Cref{def:prom-perms} is a simple and memorable description of promotion permutations, our arguments below require a different, though equivalent, description. The following is directly adapted from \cite[\S3-\S5]{fluctuating-paper}, specialized to the case of rectangular standard tableaux. See \Cref{sec:background-Fomin} for historical remarks and further context.

Write $\lambda \to \nu$ to mean that $\lambda$ and $\nu$ are partitions where $\nu$ is obtained from $\lambda$ by adding a single cell. Recall that $T \in \SYT(\lambda)$ can be thought of as a sequence of partitions $\varnothing = \lambda^0 \to \lambda^1 \to \cdots \to \lambda^n = \lambda$, where $\lambda^i$ is the shape of the tableau obtained from $T$ by removing all cells filled with entries greater than $i$.

\begin{definition}\label{def:local-rules}
  A \emph{(decorated) local rule diagram} is a square
  \begin{equation}\label{eq:local_rule}
  \begin{tikzcd}
  \lambda \ar{r}[name=U]{}
    & \nu \\
  \kappa \uar{} \ar{r}[swap,name=D]{}
    & \mu \ar{u}[swap]{}
  \ar[to path={(U) node[midway] {$(i)$}  (D)}]{}
  \end{tikzcd}
  \end{equation}
  where
  \begin{equation}\label{eq:local_rule.sort}
    \mu = \sort(\nu + \kappa - \lambda) \qquad\text{and}\qquad
      \lambda = \sort(\nu+\kappa-\mu),
  \end{equation}
  with pointwise addition and subtraction. The optional decoration $i$ is present if and only if $\nu + \kappa - \lambda$ is not already sorted, in which case sorting is accomplished by applying a unique simple transposition $s_i = (i\ i+1)$ for $1 \leq i \leq r-1$.

  A \emph{local rule} fills in a missing lower right or upper left corner in \eqref{eq:local_rule} with $\mu$ or $\lambda$ as determined by \eqref{eq:local_rule.sort}. In fact, the two conditions of \eqref{eq:local_rule.sort} imply each other, so local rules result in local rule diagrams.
\end{definition}

\begin{remark}\label{rem:sorting-cases}
    In such a local rule diagram, there are three possibilities: (i) $\nu$ differs from $\kappa$ by two disconnected cells; (ii) $\nu$ differs from $\kappa$ by two adjacent cells in the same row; and (iii) $\nu$ differs from $\kappa$ by two adjacent cells in the same column. In case (i), $\lambda \neq \mu$ are obtained by adding the two cells in either order. In cases (ii)-(iii), $\lambda = \mu$ is obtained by adding the first of the two cells. Only in case (iii) does sorting occur, and in this case the two cells appear in rows $i, i+1$.
\end{remark}

\begin{definition}
  The \textit{promotion-evacuation diagram} of a rectangular standard tableau
    \[ T = \varnothing \to \lambda^1 \to \cdots \to \lambda^{n-1} \to r \times c \]
  is obtained by iteratively applying local rules to fill in the parallelogram
  \begin{equation*}
  \PEdiagram(T) \coloneqq
  \begin{tikzcd}[column sep=scriptsize, scale cd=0.8]
    \varnothing \rar
      & \lambda^1 \rar
      & \cdots \rar
      & \lambda^{n-1} \rar
      & r \times c \\
    \ 
      & \varnothing \rar[dashed]{} \uar
      & \cdots \rar[dashed]{} \uar[dashed]{}
      & \cdot \rar[dashed]{} \uar[dashed]{}
      & \cdot \rar[dashed]{} \uar[dashed]{}
      & r \times c \\
    \ 
      & \ 
      & \ddots \rar[dashed]{} \uar{}
      & \vdots \rar[dashed]{} \uar[dashed]{}
      & \vdots \rar[dashed]{} \uar[dashed]{}
      & \cdot \rar[dashed]{} \uar[dashed]{}
      & \ddots \\
    \ 
      & \ 
      & \ 
      & \varnothing \rar[dashed]{} \uar{}
      & \cdot \rar[dashed]{} \uar[dashed]{}
      & \cdots \rar[dashed]{} \uar[dashed]{}
      & \cdot \rar[dashed]{} \uar[dashed]{}
      & r \times c \\
    \ 
      & \ 
      & \ 
      & \ 
      & \varnothing \rar[dashed]{} \uar{}
      & \cdot \rar[dashed]{} \uar[dashed]{}
      & \cdots \rar[dashed]{} \uar[dashed]{}
      & \cdot \rar[dashed]{} \uar[dashed]{}
      & r \times c \\
  \end{tikzcd}
  \end{equation*}
  starting with the given first row and working your way from upper left to lower right.
\end{definition}

\begin{remark}\label{rem:PEdiagram-pieces}
    The $j$-th row of $\PEdiagram(T)$ encodes $\promotion^j(T)$, starting with $j=0$ at the top. The $n$-th column encodes $\evacuation(T)$ when read bottom-to-top. Since $\promotion^n(T) = T$, both the first and last rows encode $T$.
\end{remark}

\begin{example}\label{ex:PEdiagram}
   The decorated promotion-evacuation diagram of $T = \vcenter{\hbox{\ytableaushort{13,25,46}}}$ from \Cref{ex:prom-dihedral} is
  \[
  \setlength\arraycolsep{2pt}
  \begin{tikzcd}[ampersand replacement=\&,column sep=tiny,row sep=tiny,scale cd=0.67]
    \varnothing \rar[""{name=L00r}]
      \& \ytableaushort{\ } \rar[""{name=L01r}]
      \& \ytableaushort{\ ,\ } \rar[""{name=L02r}]
      \& \ytableaushort{\ \ ,\ } \rar[""{name=L03r}]
      \& \ytableaushort{\ \ ,\ ,\ } \rar[""{name=L04r}]
      \& \ytableaushort{\ \ ,\ \ ,\ } \rar[""{name=L05r}]
      \& \ytableaushort{\ \ ,\ \ ,\ \ } \rar[phantom,""{name=L06r}]
      \& \ \\
    \ \rar[phantom,""{name=L10r}]
      \& \varnothing \rar[""{name=L11r}] \uar
      \& \ytableaushort{\ } \rar[""{name=L12r}] \uar
      \& \ytableaushort{\ \ } \rar[""{name=L13r}] \uar
      \& \ytableaushort{\ \ ,\ } \rar[""{name=L14r}] \uar
      \& \ytableaushort{\ \ ,\ \ } \rar[""{name=L15r}] \uar
      \& \ytableaushort{\ \ ,\ \ ,\ } \rar[""{name=L16r}] \uar
      \& \ytableaushort{\ \ ,\ \ ,\ \ } \rar[phantom,""{name=L17r}]
      \& \ \\
    \ 
      \& \ \rar[phantom,""{name=L21r}]
      \& \varnothing \rar[""{name=L22r}] \uar
      \& \ytableaushort{\ } \rar[""{name=L23r}] \uar
      \& \ytableaushort{\ ,\ } \rar[""{name=L24r}] \uar
      \& \ytableaushort{\ \ ,\ } \rar[""{name=L25r}] \uar
      \& \ytableaushort{\ \ ,\ ,\ } \rar[""{name=L26r}] \uar
      \& \ytableaushort{\ \ ,\ \ ,\ } \rar[""{name=L27r}] \uar
      \& \ytableaushort{\ \ ,\ \ ,\ \ } \rar[phantom,""{name=L28r}]
      \& \ \\
    \ 
      \& \ 
      \& \ \rar[phantom,""{name=L32r}]
      \& \varnothing \rar[""{name=L33r}] \uar
      \& \ytableaushort{\ } \rar[""{name=L34r}] \uar
      \& \ytableaushort{\ \ } \rar[""{name=L35r}] \uar
      \& \ytableaushort{\ \ ,\ } \rar[""{name=L36r}] \uar
      \& \ytableaushort{\ \ ,\ \ } \rar[""{name=L37r}] \uar
      \& \ytableaushort{\ \ ,\ \ ,\ } \rar[""{name=L38r}] \uar
      \& \ytableaushort{\ \ ,\ \ ,\ \ } \rar[phantom,""{name=L39r}]
      \& \ \\
    \ 
      \& \ 
      \& \ 
      \& \ \rar[phantom,""{name=L43r}]
      \& \varnothing \rar[""{name=L44r}] \uar
      \& \ytableaushort{\ } \rar[""{name=L45r}] \uar
      \& \ytableaushort{\ ,\ } \rar[""{name=L46r}] \uar
      \& \ytableaushort{\ \ ,\ } \rar[""{name=L47r}] \uar
      \& \ytableaushort{\ \ ,\ ,\ } \rar[""{name=L48r}] \uar
      \& \ytableaushort{\ \ ,\ \ ,\ } \rar[""{name=L49r}] \uar
      \& \ytableaushort{\ \ ,\ \ ,\ \ } \rar[phantom,""{name=L4Ar}]
      \& \ \\
    \ 
      \& \ 
      \& \ 
      \& \ 
      \& \ \rar[phantom,""{name=L54r}]
      \& \varnothing \rar[""{name=L55r}] \uar
      \& \ytableaushort{\ } \rar[""{name=L56r}] \uar
      \& \ytableaushort{\ \ } \rar[""{name=L57r}] \uar
      \& \ytableaushort{\ \ ,\ } \rar[""{name=L58r}] \uar
      \& \ytableaushort{\ \ ,\ \ } \rar[""{name=L59r}] \uar
      \& \ytableaushort{\ \ ,\ \ ,\ } \rar[""{name=L5Ar}] \uar
      \& \ytableaushort{\ \ ,\ \ ,\ \ } \rar[phantom,""{name=L5Br}]
      \& \ \\
    \ 
      \& \ 
      \& \ 
      \& \ 
      \& \ 
      \& \ \rar[phantom,""{name=L65r}]
      \& \varnothing \rar[""{name=L66r}] \uar
      \& \ytableaushort{\ } \rar[""{name=L67r}] \uar
      \& \ytableaushort{\ ,\ } \rar[""{name=L68r}] \uar
      \& \ytableaushort{\ \ ,\ } \rar[""{name=L69r}] \uar
      \& \ytableaushort{\ \ ,\ ,\ } \rar[""{name=L6Ar}] \uar
      \& \ytableaushort{\ \ ,\ \ ,\ } \rar[""{name=L6Br}] \uar
      \& \ytableaushort{\ \ ,\ \ ,\ \ }.
      \arrow[phantom,from=L00r,to=L10r,"{}"]
      \arrow[phantom,from=L01r,to=L11r,"{1}"]
      \arrow[phantom,from=L02r,to=L12r,"{}"]
      \arrow[phantom,from=L03r,to=L13r,"{2}"]
      \arrow[phantom,from=L04r,to=L14r,"{}"]
      \arrow[phantom,from=L05r,to=L15r,"{}"]
      \arrow[phantom,from=L06r,to=L16r,"{}"]
      \arrow[phantom,from=L11r,to=L21r,"{}"]
      \arrow[phantom,from=L12r,to=L22r,"{}"]
      \arrow[phantom,from=L13r,to=L23r,"{}"]
      \arrow[phantom,from=L14r,to=L24r,"{1}"]
      \arrow[phantom,from=L15r,to=L25r,"{}"]
      \arrow[phantom,from=L16r,to=L26r,"{2}"]
      \arrow[phantom,from=L17r,to=L27r,"{}"]
      \arrow[phantom,from=L22r,to=L32r,"{}"]
      \arrow[phantom,from=L23r,to=L33r,"{1}"]
      \arrow[phantom,from=L24r,to=L34r,"{}"]
      \arrow[phantom,from=L25r,to=L35r,"{2}"]
      \arrow[phantom,from=L26r,to=L36r,"{}"]
      \arrow[phantom,from=L27r,to=L37r,"{}"]
      \arrow[phantom,from=L28r,to=L38r,"{}"]
      \arrow[phantom,from=L33r,to=L43r,"{}"]
      \arrow[phantom,from=L34r,to=L44r,"{}"]
      \arrow[phantom,from=L35r,to=L45r,"{}"]
      \arrow[phantom,from=L36r,to=L46r,"{1}"]
      \arrow[phantom,from=L37r,to=L47r,"{}"]
      \arrow[phantom,from=L38r,to=L48r,"{2}"]
      \arrow[phantom,from=L39r,to=L49r,"{}"]
      \arrow[phantom,from=L44r,to=L54r,"{}"]
      \arrow[phantom,from=L45r,to=L55r,"{1}"]
      \arrow[phantom,from=L46r,to=L56r,"{}"]
      \arrow[phantom,from=L47r,to=L57r,"{2}"]
      \arrow[phantom,from=L48r,to=L58r,"{}"]
      \arrow[phantom,from=L49r,to=L59r,"{}"]
      \arrow[phantom,from=L4Ar,to=L5Ar,"{}"]
      \arrow[phantom,from=L55r,to=L65r,"{}"]
      \arrow[phantom,from=L56r,to=L66r,"{}"]
      \arrow[phantom,from=L57r,to=L67r,"{}"]
      \arrow[phantom,from=L58r,to=L68r,"{1}"]
      \arrow[phantom,from=L59r,to=L69r,"{}"]
      \arrow[phantom,from=L5Ar,to=L6Ar,"{2}"]
      \arrow[phantom,from=L5Br,to=L6Br,"{}"]
  \end{tikzcd}
  \]
  A decoration $i$ appears precisely when the corresponding southwest and northeast partitions differ by two adjacent cells in rows $i, i+1$. We again see that this example satisfies $\promotion^2(T) = T$.
\end{example}

\begin{definition}
  Let $n = rc$. The \emph{promotion matrix} of $T \in \SYT(r \times c)$ is the $n \times n$ matrix $\PM(T)$ obtained as follows. Place the decorations of $\PEdiagram(T)$ in a matrix by wrapping columns around modulo $n$, where entries without decorations are $0$, and each row has an initial $0$.
\end{definition}

\begin{example}\label{ex:prom-matrix}
  Continuing \Cref{ex:PEdiagram},
  \[
  \PM\left(\,\vcenter{\hbox{\ytableaushort{13,25,46}}}\,\right)
  =
  \begin{bmatrix}
    \cdot & 1 & \cdot & 2 & \cdot & \cdot \\
    2 & \cdot & \cdot & \cdot & 1 & \cdot \\
    \cdot & \cdot & \cdot & 1 & \cdot & 2 \\
    1 & \cdot & 2 & \cdot & \cdot & \cdot \\
    \cdot & 2 & \cdot & \cdot & \cdot & 1 \\
    \cdot & \cdot & 1 & \cdot & 2 & \cdot \\
  \end{bmatrix}\]
  Here we write $\cdot$ instead of $0$ for visual clarity.
\end{example}

\begin{remark}\label{rem:prom-structure}
    For $T \in \SYT(r \times c)$, the promotion matrix $\PM(T)$ is highly structured.
    \begin{itemize}
        \item The $n \coloneqq rc$ entries on the main diagonal are zero.
        \item Every row and column has entries $1 \cdots (r-1)$ listed in increasing order when reading cyclically going east or north starting from the main diagonal and skipping $0$'s.
        \item An entry $i$ has an entry $r-i$ in the transposed position.
    \end{itemize}
\end{remark}

Promotion matrices directly encode all promotion permutations simultaneously. In order to be consistent with promotion-evacuation diagrams, we consider a permutation $\rho \in \mathfrak{S}_n$ as the $n \times n$ matrix $M$ with $M_{i,\rho(i)} = 1$ for $1 \leq i \leq n$ and $0$'s elsewhere.

\begin{proposition}[{\cite[Prop.~5.19]{fluctuating-paper}}]\label{prop:prom-matrix}
    Let $T \in \SYT(r \times c)$. Then
      \[ \PM(T) = \sum_{i=1}^{r-1} i \prom_i(T). \]
\end{proposition}

\begin{example}
    When $T = 13/25/46$, by \Cref{ex:prom-matrix} and \Cref{prop:prom-matrix}, $\prom_1(T) = [2,5,4,1,6,3]$ and $\prom_2(t) = [4,1,6,3,2,5]$, as in \Cref{ex:prom-dihedral}. The $24 \times 24$ promotion matrix associated with $S = \stack(P, Q)$ from \Cref{ex:prom-RS} is in \Cref{ex:main} below.
\end{example}

Finding a direct characterization of which matrices are promotion matrices is an open problem \cite[Rem.~6.15]{fluctuating-paper}.

\begin{problem}\label{prob:prom-perms}
    Give a direct combinatorial characterization of which tuples $(\rho_1, \ldots, \rho_{r-1}) \in \mathfrak{S}_n^{r-1}$ are of the form $\prom_\bullet(T)$ for some $T \in \SYT(r \times c)$, where $n \coloneqq rc$.
\end{problem}

\noindent When $r=2$, they correspond precisely to non-crossing perfect matchings by standard Catalan combinatorics. \Cref{cor:prom-RS} provides a partial answer to \Cref{prob:prom-perms} in a very special case but for arbitrary even $r$.

\subsection{Existing work and Fomin's local rules}\label{sec:background-Fomin}

We briefly compare our construction in \Cref{sec:background-local-rules} to existing work. This material is not required for what follows, but it places our work in context.

Fomin \cite{Fomin-RSK} introduced growth diagrams and a version of local rules to compute the Robinson--Schensted--Knuth correspondence. See, e.g., \cite[Fig.~10]{Krattenthaler-growth-diagrams} for an example. Fomin's local rules, as laid out in \cite{Krattenthaler-growth-diagrams}, allow edges $\lambda \to \nu$ which add either $0$ or $1$ boxes and include an optional decoration $\times$. Moreover, Fomin's local rules work perpendicularly to those described in \Cref{def:local-rules}, namely they fill in the northeast corner $\nu$ (or the southwest corner $\kappa$) given the other three partitions and the decoration. In fact, Fomin's local rules may be written in terms of sorting very similarly to \Cref{def:local-rules}, which was observed by Pfannerer--Rubey--Westbury \cite[Fig.~3]{Pfannerer-Rubey-Westbury}. This style of sorting rule was originally observed by van Leeuwen \cite[Rule~4.1.1]{vanLeeuwen}.

Various authors have used constructions similar to or generalizing Fomin's local rules to give analogues of the RS correspondence. See, e.g., \cite{Krattenthaler-growth-diagrams} for growth diagrams of arbitrary Ferrers shape, work including the first author (which motivated the present work) in \cite{Papper.Pfannerer.Schilling.Simone,Pfannerer-Rubey-Westbury} for $r$-fans of Dyck paths and oscillating tableaux, respectively, and \cite{Akhmejanov} which gives a refinement of the classic type $A$ version and an interpretation in terms of the geometric Satake correspondence.

We note that \Cref{thm:prom-RS} does not follow from this existing work and is in some sense perpendicular to it. In addition to the direction local rules are applied, one key difference in the promotion permutations of \cite{fluctuating-paper} is the use of multiple decorations rather than a single decoration. Instead of using Fomin's local rules and the resulting construction of RS, ours uses the local rules in \Cref{def:local-rules} and, by \Cref{thm:main}, encodes Viennot's geometric shadow line construction.

\subsection{Crossing and nesting numbers}

Let $M$ be a \emph{perfect matching} on $[2n]$, i.e., a set partition whose blocks are all of size $2$. Equivalently, we may think of $M$ as a fixed-point free involution $\rho$ in $\mathfrak{S}_{2n}$ where the cycles of $\rho$ are given by the blocks of $M$. An example is depicted in \Cref{fig:matching-alt}.

\begin{definition}[{See, e.g., \cite[\S3]{Krattenthaler-growth-diagrams} and the citations therein}]
    \ 
    \begin{itemize}
        \item A \emph{$k$-crossing} of $M$ is a collection of $k$ blocks $\{i_1 < j_1\}, \ldots, \{i_k < j_k\}$ such that $i_1 < \cdots < i_k < j_1 < \cdots < j_k$. Pictorially, in the corresponding chord diagram all $k$ chords must cross pairwise. The \emph{crossing number} of $M$ is the maximum $k$ for which $M$ has a $k$-crossing.
        \item A \emph{$k$-nesting} of $M$ is a collection of $k$ blocks $\{i_1 < j_1\}, \ldots, \{i_k < j_k\}$ such that $i_1 < \cdots < i_k < j_k < \cdots < j_1$. Pictorially, in the corresponding chord diagram, the $k$ chords nest one inside the next. The \emph{nesting number} of $M$ is the maximum $k$ for which $M$ has a $k$-nesting.
    \end{itemize}
\end{definition}

%
%
%

\subsection{Viennot's construction}\label{sec:background-Viennot}
We now recall Viennot's \emph{geometric shadow line construction} for the Robinson--Schensted correspondence (see \cite{Viennot-geometric}, \cite[\S3.6]{Sagan-symmetric-group}). One slightly unusual feature of our exposition is that we must state the construction and its relation to the RS correspondence in the generality of light shining from any corner. We also note that our convention for permutation matrices in \Cref{sec:background-local-rules} differs from the common convention used in \cite[\S3.6]{Sagan-symmetric-group}.

Define the following partial orders on the plane $\mathbb{Z} \times \mathbb{Z}$:
\begin{multicols}{2}
\begin{itemize}
    \item $(a, b) \nearrow (x, y)$ if and only if $a \leq x$ and $b \leq y$
    \item $(a, b) \swarrow (x, y)$ if and only if $a \geq x$ and $b \geq y$
    \item $(a, b) \searrow (x, y)$ if and only if $a \leq x$ and $b \geq y$
    \item $(a, b) \nwarrow (x, y)$ if and only if $a \geq x$ and $b \leq y$.
\end{itemize}
\end{multicols}
\noindent To understand $\nearrow$ pictorially, we imagine there is a bright light shining to the northeast in such a way that $(a, b)$ casts a shadow which covers all points weakly north or east of $(a, b)$. Then $(a, b) \nearrow (x, y)$ if and only if $(x, y)$ is in the shadow of $(a, b)$. The rest are similar.

In what follows, suppose $\mathcal{P} = \{(x_1, y_1), \ldots, (x_k, y_k)\}$ is a collection of points in $\mathbb{Z} \times \mathbb{Z}$, where $x_i \neq x_j$ and $y_i \neq y_j$ for $i \neq j$. In the next two definitions, we assume our light shines in the $\nearrow$ direction for concreteness, though the rest are verbatim identical upon replacing the symbol $\nearrow$ with $\searrow, \swarrow, \nwarrow$.

\begin{definition}
    The \emph{first shadow line} $\mathcal{L}_1(\mathcal{P}, \nearrow)$ is the set of minimal elements of $\mathcal{P}$ under $\nearrow$. Pictorially, we identify this with the border of the region between light and shadow cast by $\mathcal{P}$ with respect to $\nearrow$; see \Cref{fig:Viennot-construction} (left). Inductively, the \emph{$j$-th shadow line}, $\mathcal{L}_j(\mathcal{P}, \nearrow)$, is the set of minimal elements of $\mathcal{P} - \cup_{k=1}^{j-1} \mathcal{L}_j(\mathcal{P}, \nearrow)$ under $\nearrow$.
\end{definition}

\begin{definition}
    The \emph{skeleton} of a shadow line $\mathcal{L}$ with respect to $\nearrow$, written $\mathcal{S}(\mathcal{L}, \nearrow)$, is the set of minimal elements under $\nearrow$ of the region consisting of all points $(x, y)$ covered by at least two shadows from $\mathcal{L}$. If $\mathcal{P}$ has shadow lines $\mathcal{L}_1, \mathcal{L}_2, \ldots$, the \emph{skeleton} of $\mathcal{P}$ is $\mathcal{S}(\mathcal{P}, \nearrow) \coloneqq \cup_i \mathcal{S}(\mathcal{L}_i, \nearrow)$; see \Cref{fig:Viennot-construction} (left). The \emph{$i$-th skeleton} of $\mathcal{P}$ is defined recursively by $\mathcal{S}^0(\mathcal{P}, \nearrow) = \mathcal{P}$ and $\mathcal{S}^{i+1}(\mathcal{P}, \nearrow) = \mathcal{S}(\mathcal{S}^i(\mathcal{P}, \nearrow), \nearrow)$; see \Cref{fig:Viennot-construction} (right).
\end{definition}

\begin{example}
    Let $\mathcal{P} = \mathcal{P}^0 = \{(1, 6), (2, 2), (3, 4), (4, 5), (5, 3), (6, 1), (7, 7)\}$, which are the filled points in \Cref{fig:Viennot-construction} (left). That figure depicts the four shadow lines \[\{\mathcal{L}_j(\mathcal{P}, \nearrow)\} =  \{\{(1, 6), (2, 2), (6, 1)\}, \{(3, 3)\}, \{(4, 5), (5, 4)\}, \{(7, 7)\}\}.\] The skeleton of $\mathcal{P}$ consists of the unfilled points, $\mathcal{S}^1(\mathcal{P}, \nearrow) = \{(2, 6), (5, 4), (6, 2)\}$. Computing further skeleta gives $\mathcal{S}^2(\mathcal{P}, \nearrow) = \{(5, 6), (6, 5)\}$ and $\mathcal{S}^3(\mathcal{P}, \nearrow) = \{(6, 6)\}$. All shadow lines and all $i$-skeleta are drawn together in \Cref{fig:Viennot-construction} (right), using different colors for each $i$, where shadow lines associated with the $i$-skeleton are labeled with $i+1$.
\end{example}

\begin{figure}[hbtp]
    \noindent
    \resizebox{\textwidth}{!}{%
    \begin{tikzpicture}[>=stealth]

    \begin{scope}[xshift=0cm]

    \foreach \x in {1,...,7}{
        \draw[line width=0.6pt] (\x,0) -- (\x,8);
        \node[below] at (\x,0) {\small \x};
    }
    \foreach \y in {1,...,7}{
        \draw[line width=0.6pt] (0,\y) -- (8,\y);
        \node[left] at (0,\y) {\small \y};
    }

    \draw[->, line width=1.5pt] (-0.4,-0.4) -- (0.4,0.4);

    \fill[lightgray, opacity=0.4] 
        (1,8) -- (1,6) -- (2,6) -- (2,2) -- (6,2) -- (6,1) -- (8,1) --
        (8,8) -- cycle;

    \foreach \x/\y in {1/6,2/2,3/4,4/5,5/3,6/1,7/7}{
        \fill (\x,\y) circle (6pt);
    }

    \foreach \x/\y in {2/6,5/4,6/2}{
        \draw[line width=1pt] (\x,\y) circle (6pt);
    }

    \draw[line width=3pt] (1,8) -- (1,6) -- (2,6) -- (2,2) -- (6,2) -- (6,1) -- (8,1);
    \draw[line width=3pt] (3,8) -- (3,4) -- (5,4) -- (5,3) -- (8,3);
    \draw[line width=3pt] (4,8) -- (4,5) -- (8,5);
    \draw[line width=3pt] (8,7) -- (7,7) -- (7,8);

    \end{scope}

    \begin{scope}[xshift=11cm]

    \foreach \x in {1,...,7}{
        \draw[line width=0.6pt] (\x,0) -- (\x,8);
        \node[below] at (\x,0) {\small \x};
    }
    \foreach \y in {1,...,7}{
        \draw[line width=0.6pt] (0,\y) -- (8,\y);
        \node[left] at (0,\y) {\small \y};
    }

    \draw[->, line width=1.5pt] (-0.4,-0.4) -- (0.4,0.4);

    \draw[line width=3pt, customred] (1,8) -- (1,6) -- (2,6) -- (2,2) -- (6,2) -- (6,1) -- (8,1)
        node[above] at (1,8) {1} node[right] {1};
    \draw[line width=3pt, customred] (3,8) -- (3,4) -- (5,4) -- (5,3) -- (8,3)
        node[above] at (3,8) {1} node[right] {1};
    \draw[line width=3pt, customred] (4,8) -- (4,5) -- (8,5)
        node[above] at (4,8) {1} node[right] {1};
    \draw[line width=3pt, customred] (7,8) -- (7,7) -- (8,7)
        node[above] at (7,8) {1} node[right] {1};

    \draw[line width=3pt, orange] (2,8) -- (2,6) -- (5,6) -- (5,4) -- (6,4) -- (6,2) -- (8,2)
        node[above] at (2,8) {2} node[right] {2};

    \draw[line width=3pt, green!70!black] (5,8) -- (5,6) -- (6,6) -- (6,4) -- (8,4)
        node[above] at (5,8) {3} node[right] {3};

    \draw[line width=3pt, blue] (6,8) -- (6,6) -- (8,6)
        node[above] at (6,8) {4} node[right] {4};

    \node[right] at (9,4) {$P = \ytableaushort{1357,2,4,6}$};
    \node[above] at (4,9) {$Q = \ytableaushort{1347,2,5,6}$};

    \end{scope}

    \end{tikzpicture}%
    }
    \caption{(\textsc{Left}) A collection $\mathcal{P}$ depicted as filled points, together with the shadows cast with respect to light shining $\nearrow$. The shadow lines partition $\mathcal{P}$ into four pieces. The skeleton $\mathcal{S}(\mathcal{P}, \nearrow)$ is depicted as unfilled circles. (\textsc{Right}) The union of the shadow lines for all skeleta of $\mathcal{P}$ drawn simultaneously. Each $i$-skeleton $\mathcal{S}^i$ for $0 \leq i \leq 3$ is in a different color and is labeled at the right and top with $i+1$. Reading these labels gives lattice words $w_P = 1213141$, $w_Q = 1211341$. These words encode successive row indexes of the indicated RS tableaux $(P, Q)$, which corresponds under RS to the permutation $\rho = [6, 2, 4, 5, 3, 1, 7]$.}
    \label{fig:Viennot-construction}
\end{figure}
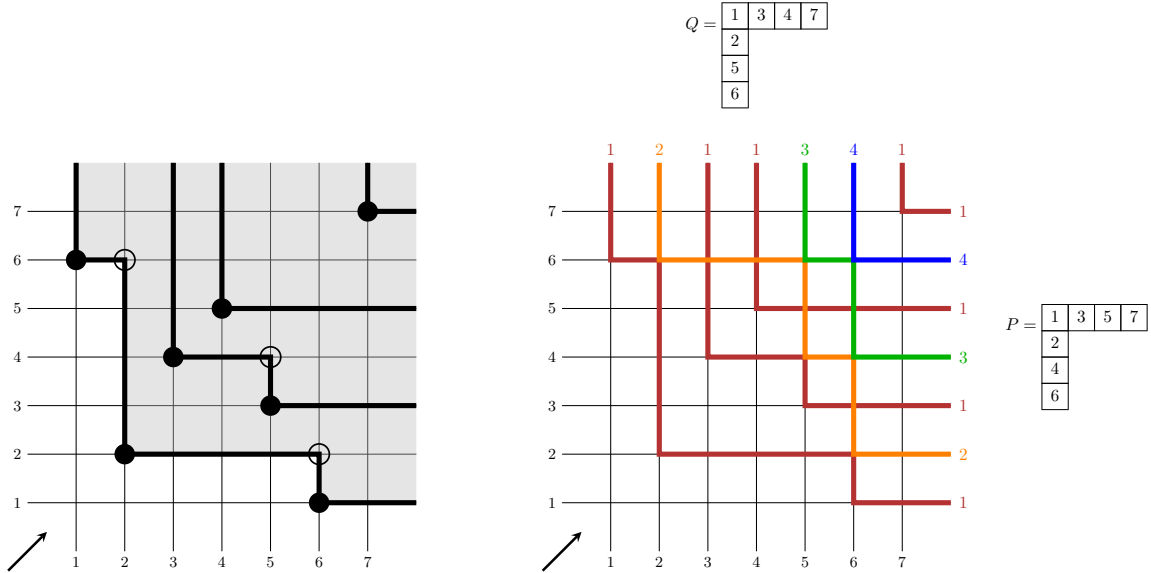

Recall that the \emph{lattice word} of a standard tableau $T \in \SYT(\lambda)$ is the word $L(T) = w_1 \cdots w_n$ where $w_i=r$ if the entry $i$ in $T$ appears in row $r$. See, e.g., \cite[\S2.4, \S8]{fluctuating-paper} for more information.

\begin{definition}
    Consider the union of the shadow lines of all the $i$-skeleta $\mathcal{P} = \mathcal{S}^0(\mathcal{P}, \nearrow), \mathcal{S}^1(\mathcal{P}, \nearrow), \ldots$ with respect to $\nearrow$. Label shadow lines coming from the $i$-skeleton with $i+1$ and let $w_P(\mathcal{P}, \nearrow)$ be the word obtained by reading the sequence of numbers along the extreme right edge of the union from bottom to top; see \Cref{fig:Viennot-construction} (right). Let $w_Q(\mathcal{P}, \nearrow)$ be the word obtained similarly by reading the extreme top edge of the union from left to right. It is not difficult to see that $w_P$ and $w_Q$ are lattice words, hence they correspond to some standard tableaux $P, Q$, respectively. In this case, write the output of Viennot's construction as $\mathcal{V}(\mathcal{P}, \nearrow) = (P, Q)$.
\end{definition}

The main result concerning Viennot's construction is that it computes the RS correspondence:

\begin{theorem}[{\cite{Viennot-geometric}}]\label{thm:Viennot-basic}
    Let $\rho \in \mathfrak{S}_n$ and let $\mathcal{P} = \{(i, \rho_i) \mid 1 \leq i \leq n\}$. Then $\RS(\rho) = \mathcal{V}(\mathcal{P}, \nearrow)$.
\end{theorem}

Define $\mathcal{V}(\mathcal{P}, -) = (P, Q)$ with respect to any of the directions $\nearrow, \searrow, \swarrow, \nwarrow$ so that rotating both $\mathcal{P}$ and the direction of the light together preserves $(P, Q)$. Explicitly:
\begin{itemize}
    \item[($\nearrow$)] Read $P$ bottom-to-top along the right edge, $Q$ left-to-right along the top edge;
    \item[($\searrow$)] Read $P$ left-to-right along the bottom edge, $Q$ top-to-bottom along the right edge;
    \item[($\swarrow$)] Read $P$ bottom-to-top along the left edge, $Q$ right-to-left along the bottom edge;
    \item[($\nwarrow$)] Read $P$ right-to-left along the top edge, $Q$ bottom-to-top along the left edge.  
\end{itemize}

We require additional, slightly less well-known symmetries of Viennot's construction and the RS correspondence, which essentially go back to Schensted and Knuth. See \cite[\S4.1]{vanLeeuwen} for a thorough discussion and further references.

\begin{theorem}\label{thm:Viennot-symmetries}
    Let $\mathcal{P}$ be a collection of points as above. The following are equivalent.
    \begin{enumerate}[(i)]
        \item $\mathcal{V}(\mathcal{P}, \nearrow) = (P, Q)$
        \item $\mathcal{V}(\mathcal{P}, \searrow) = (Q^\top, \evacuation(P)^\top)$
        \item $\mathcal{V}(\mathcal{P}, \swarrow) = (\evacuation(P), \evacuation(Q))$
        \item $\mathcal{V}(\mathcal{P}, \nwarrow) = (\evacuation(Q)^\top, P^\top)$.
    \end{enumerate}
\end{theorem}

\begin{proof}
    Translating $\mathcal{P}$ and removing unused rows and columns does not change the output of $\mathcal{V}$, so we may assume $\mathcal{P}$ comes from a permutation $\rho$ as in \Cref{thm:Viennot-basic}. Let $\mathcal{P}'$ be obtained from $\mathcal{P}$ by rotating counterclockwise $90^\circ$. With the convention for permutations in \Cref{thm:Viennot-basic}, $\mathcal{P}'$ corresponds to the permutation $\rho' = \rho^{-1} w_0$. Hence we have
      \[ \mathcal{V}(\mathcal{P}, \searrow) = \mathcal{V}(\mathcal{P}', \nearrow) = \RS(\rho') = \RS(\rho^{-1} w_0). \]
    Thus (i) $\Leftrightarrow$ (ii) is equivalent to $\RS(\rho) = (P, Q) \Leftrightarrow \RS(\rho^{-1} w_0) = (Q^\top, \evacuation(P)^\top)$. By symmetry, this is equivalent to $\RS(\rho) = (P, Q) \Leftrightarrow \RS(w_0 \rho) = (\evacuation(P)^\top, Q^\top)$, which follows from, e.g.,~\cite[Thm.~A1.2.10]{Stanley:EC2} and \cite[A1.2.11]{Stanley:EC2}. The remaining implications follow by repeating this step and observing that evacuation and tranposition are commuting involutions.
\end{proof}

\section{Main theorem and proof}\label{sec:main}

We now turn to the proof of \Cref{thm:prom-RS}. We in fact state and prove a stronger statement, \Cref{thm:main}, which fully describes the northeast and southwest quadrants of the promotion matrix of a stacked rectangular tableau in terms of the skeleta of the corresponding RS permutations.

\subsection{Main theorem statement}

We first recall and introduce some notation. As in \Cref{sec:background-local-rules}, a promotion permutation $\rho \in \mathfrak{S}_n$ corresponds to the $n \times n$ matrix $M = M(\rho)$ with $M_{i, \rho(i)} = 1$ for $1 \leq i \leq n$ and $0$'s elsewhere. If $M$ is a $2n \times 2n$ matrix, let $M^{\NE}$, $M^{\SE}$, $M^{\SW}$, or $M^{\NW}$ denote the northeast, southeast, southwest, or northwest $n \times n$ block of $M$, respectively. If $M$ is an $n \times n$ matrix consisting of $0$'s and $1$'s with at most one non-zero entry in each row and column, the corresponding subset of $[n] \times [n]$ in Cartesian coordinates is $\mathcal{P}(M) \coloneqq \{(y, n+1-x) \mid M_{xy}=1\}$. If $\rho \in \mathfrak{S}_{2n}$, write $\mathcal{P}^{\NE}(\rho) \coloneqq \mathcal{P}(M(\rho)^{\NE})$ for the subset of $[n] \times [n]$ corresponding to the northeast block of the permutation matrix of $\rho$. Recall from \Cref{sec:background-Viennot} that $\mathcal{V}(\mathcal{P}, \nearrow)$ refers to the output of Viennot's construction with light shining $\nearrow$, and $\mathcal{S}^j(\mathcal{P}, \nearrow)$ is the result of iteratively computing the skeleton of $\mathcal{P}$, $j$ times.

\begin{theorem}\label{thm:main}
    Let $S = \stack(P, Q) \in \SYT^{\stacked}(2r \times c)$ where $P, Q \in \SYT(r \times c)$ and $n \coloneqq rc$. Let $\rho_i \coloneqq \prom_i(S) \in \mathfrak{S}_{2n}$ for $1 \leq i \leq 2r-1$. Then
    \begin{equation}\label{eq:main-1}
        \mathcal{V}(\mathcal{P}^{\NE}(\rho_r), \nearrow) = (\evacuation(P), Q).
    \end{equation}
    Furthermore, for $0 \leq k \leq r-1$,
    \begin{equation}\label{eq:main-2}
    \begin{split}
        \mathcal{P}^{\NE}(\rho_{r+k}) &= \mathcal{S}^k(\mathcal{P}^{\NE}(\rho_r), \nearrow) \\
        \mathcal{P}^{\NE}(\rho_{r-k}) &= \mathcal{S}^k(\mathcal{P}^{\NE}(\rho_r), \swarrow).
    \end{split}
    \end{equation}
    More precisely, the skeleton of the $j$-th shadow line of the skeleton $\mathcal{S}^k(\mathcal{P}^{\NE}(\rho_r), \nearrow)$ (or $\mathcal{S}^k(\mathcal{P}^{\NE}(\rho_r), \swarrow)$) is the $j$-th shadow line of the next skeleton $\mathcal{S}^{k+1}(\mathcal{P}^{\NE}(\rho_r), \nearrow)$ (or $\mathcal{S}^{k+1}(\mathcal{P}^{\NE}(\rho_r), \swarrow)$).
\end{theorem}

We prove \Cref{thm:main} in \Cref{sec:main-proof}. Before that, we discuss the structure of $\PM(S)$ in general, give an example, and deduce the consequences of \Cref{thm:main} from \Cref{sec:intro}.

\begin{remark}\label{rem:main-structure}
    \Cref{thm:main} completely determines $\PM(S)^{\NE}$ by \Cref{rem:prom-structure}. By \Cref{thm:prom-perm-symmetries}(d) this completely determines $\PM(S)^{\SW}$. Indeed, we see the upper triangles of $\PM(S)^{\NW}$ and $\PM(S)^{\SE}$ have only entries $< r$ and their lower triangles have only entries $> r$. It follows from \Cref{sec:background-local-rules} that $\PM^{\NW}(S)$ depends only on $P$, that $\PM^{\SE}(S)$ depends only on $Q$, and that these portions of $\PM(S)$ may be computed directly from the decorated evacuation diagrams of $P$ and $Q$, respectively.
\end{remark}

\begin{example}\label{ex:main}
    We revisit $P, Q$ from \Cref{ex:prom-RS} and \Cref{fig:stacking-example}. Let $S = \stack(P, Q)$. One may compute
\[
\PM\left(\,\vcenter{\hbox{\ytableaushort{1356,279{11},48{10}{12},{13}{14}{15}{19},{16}{17}{18}{22},{20}{21}{23}{24}}}}\,\right) = 
\setlength{\arraycolsep}{2.5pt} 
\renewcommand{\arraystretch}{0.6} 
\scalebox{0.8}{$\left(
\begin{array}{cccccccccccc|cccccccccccc}
\cdot & 1 & \cdot & 2 & \cdot & \cdot & \cdot & \cdot & \cdot & \cdot & \cdot & \cdot & \cdot & \cdot & \cdot & \cdot & \cdot & \cdot & 3 & \cdot & \cdot & 4 & \cdot & 5 \\
5 & \cdot & \cdot & \cdot & \cdot & \cdot & \cdot & \cdot & \cdot & \cdot & 1 & \cdot & \cdot & \cdot & \cdot & \cdot & \cdot & \cdot & 2 & \cdot & \cdot & 3 & \cdot & 4 \\
\cdot & \cdot & \cdot & 1 & \cdot & \cdot & \cdot & \cdot & \cdot & \cdot & \cdot & 2 & \cdot & \cdot & 3 & \cdot & \cdot & 4 & \cdot & \cdot & \cdot & \cdot & 5 & \cdot \\
4 & \cdot & 5 & \cdot & \cdot & \cdot & \cdot & \cdot & \cdot & \cdot & \cdot & \cdot & \cdot & \cdot & \cdot & \cdot & \cdot & \cdot & 1 & \cdot & \cdot & 2 & \cdot & 3 \\
\cdot & \cdot & \cdot & \cdot & \cdot & \cdot & \cdot & \cdot & 1 & 2 & \cdot & \cdot & \cdot & 3 & \cdot & \cdot & 4 & \cdot & \cdot & \cdot & 5 & \cdot & \cdot & \cdot \\
\cdot & \cdot & \cdot & \cdot & \cdot & \cdot & 1 & 2 & \cdot & \cdot & \cdot & \cdot & 3 & \cdot & \cdot & 4 & \cdot & \cdot & \cdot & 5 & \cdot & \cdot & \cdot & \cdot \\
\cdot & \cdot & \cdot & \cdot & \cdot & 5 & \cdot & 1 & \cdot & \cdot & \cdot & \cdot & \cdot & \cdot & 2 & \cdot & \cdot & 3 & \cdot & \cdot & \cdot & \cdot & 4 & \cdot \\
\cdot & \cdot & \cdot & \cdot & \cdot & 4 & 5 & \cdot & \cdot & \cdot & \cdot & \cdot & \cdot & \cdot & 1 & \cdot & \cdot & 2 & \cdot & \cdot & \cdot & \cdot & 3 & \cdot \\
\cdot & \cdot & \cdot & \cdot & 5 & \cdot & \cdot & \cdot & \cdot & 1 & \cdot & \cdot & \cdot & 2 & \cdot & \cdot & 3 & \cdot & \cdot & \cdot & 4 & \cdot & \cdot & \cdot \\
\cdot & \cdot & \cdot & \cdot & 4 & \cdot & \cdot & \cdot & 5 & \cdot & \cdot & \cdot & \cdot & 1 & \cdot & \cdot & 2 & \cdot & \cdot & \cdot & 3 & \cdot & \cdot & \cdot \\
\cdot & 5 & \cdot & \cdot & \cdot & \cdot & \cdot & \cdot & \cdot & \cdot & \cdot & 1 & 2 & \cdot & \cdot & 3 & \cdot & \cdot & \cdot & 4 & \cdot & \cdot & \cdot & \cdot \\
\cdot & \cdot & 4 & \cdot & \cdot & \cdot & \cdot & \cdot & \cdot & \cdot & 5 & \cdot & 1 & \cdot & \cdot & 2 & \cdot & \cdot & \cdot & 3 & \cdot & \cdot & \cdot & \cdot \\
\hline
\cdot & \cdot & \cdot & \cdot & \cdot & 3 & \cdot & \cdot & \cdot & \cdot & 4 & 5 & \cdot & \cdot & \cdot & \cdot & \cdot & 1 & \cdot & \cdot & \cdot & \cdot & \cdot & 2 \\
\cdot & \cdot & \cdot & \cdot & 3 & \cdot & \cdot & \cdot & 4 & 5 & \cdot & \cdot & \cdot & \cdot & \cdot & \cdot & 1 & \cdot & \cdot & \cdot & 2 & \cdot & \cdot & \cdot \\
\cdot & \cdot & 3 & \cdot & \cdot & \cdot & 4 & 5 & \cdot & \cdot & \cdot & \cdot & \cdot & \cdot & \cdot & 1 & \cdot & \cdot & \cdot & 2 & \cdot & \cdot & \cdot & \cdot \\
\cdot & \cdot & \cdot & \cdot & \cdot & 2 & \cdot & \cdot & \cdot & \cdot & 3 & 4 & \cdot & \cdot & 5 & \cdot & \cdot & \cdot & \cdot & \cdot & \cdot & \cdot & \cdot & 1 \\
\cdot & \cdot & \cdot & \cdot & 2 & \cdot & \cdot & \cdot & 3 & 4 & \cdot & \cdot & \cdot & 5 & \cdot & \cdot & \cdot & \cdot & \cdot & \cdot & \cdot & 1 & \cdot & \cdot \\
\cdot & \cdot & 2 & \cdot & \cdot & \cdot & 3 & 4 & \cdot & \cdot & \cdot & \cdot & 5 & \cdot & \cdot & \cdot & \cdot & \cdot & \cdot & \cdot & 1 & \cdot & \cdot & \cdot \\
3 & 4 & \cdot & 5 & \cdot & \cdot & \cdot & \cdot & \cdot & \cdot & \cdot & \cdot & \cdot & \cdot & \cdot & \cdot & \cdot & \cdot & \cdot & 1 & \cdot & \cdot & 2 & \cdot \\
\cdot & \cdot & \cdot & \cdot & \cdot & 1 & \cdot & \cdot & \cdot & \cdot & 2 & 3 & \cdot & \cdot & 4 & \cdot & \cdot & \cdot & 5 & \cdot & \cdot & \cdot & \cdot & \cdot \\
\cdot & \cdot & \cdot & \cdot & 1 & \cdot & \cdot & \cdot & 2 & 3 & \cdot & \cdot & \cdot & 4 & \cdot & \cdot & \cdot & 5 & \cdot & \cdot & \cdot & \cdot & \cdot & \cdot \\
2 & 3 & \cdot & 4 & \cdot & \cdot & \cdot & \cdot & \cdot & \cdot & \cdot & \cdot & \cdot & \cdot & \cdot & \cdot & 5 & \cdot & \cdot & \cdot & \cdot & \cdot & 1 & \cdot \\
\cdot & \cdot & 1 & \cdot & \cdot & \cdot & 2 & 3 & \cdot & \cdot & \cdot & \cdot & \cdot & \cdot & \cdot & \cdot & \cdot & \cdot & 4 & \cdot & \cdot & 5 & \cdot & \cdot \\
1 & 2 & \cdot & 3 & \cdot & \cdot & \cdot & \cdot & \cdot & \cdot & \cdot & \cdot & 4 & \cdot & \cdot & 5 & \cdot & \cdot & \cdot & \cdot & \cdot & \cdot & \cdot & \cdot \\
\end{array}
\right)$}
\]
where we have used $\cdot$ in place of $0$ for visual clarity. Focus on the northeast block, $\PM(S)^{\NE}$. From the locations of the $3$'s in this block, we have
  \[ \prom_3^{\NE}(S) = [7, 10, 3, 12, 2, 1, 6, 11, 5, 9, 4, 8], \]
as in \Cref{ex:prom-RS}. Reading off the Cartesian coordinates of the $3$'s in this block gives
\begin{align*}
  \mathcal{P}^{\NE}(\prom_3(S))
    = \{&(1,7),(4,2),(8,1), \\
      &(2,8),(5,4),(9,3), \\
      &(3,10),(6,6),(11,5), \\
      &(7,12),(10,11),(12,9)\}.
\end{align*}
Here we have listed the four shadow lines in groups of three. Computing shadow lines for all higher skeleta and overlaying them on $\PM(S)^{\NE}$ results in the following:

\begin{center}
\begin{tikzpicture}[scale=0.45]
\draw[customred,very thick] (1,13)--(1,7)--(4,7)--(4,2)--(8,2)--(8,1)--(13,1) node[font=\scriptsize,above] at (1,13) {1} node[font=\scriptsize,right] {1};
\draw[customred,very thick] (2,13)--(2,8)--(5,8)--(5,4)--(9,4)--(9,3)--(13,3) node[font=\scriptsize,above] at (2,13) {1} node[font=\scriptsize,right] {1};
\draw[customred,very thick] (3,13)--(3,10)--(6,10)--(6,6)--(11,6)--(11,5)--(13,5) node[font=\scriptsize,above] at (3,13) {1} node[font=\scriptsize,right] {1};
\draw[customred,very thick] (7,13)--(7,12)--(10,12)--(10,11)--(12,11)--(12,9)--(13,9) node[font=\scriptsize,above] at (7,13) {1} node[font=\scriptsize,right] {1};

\draw[orange,very thick] (4,13)--(4,7)--(8,7)--(8,2)--(13,2) node[font=\scriptsize,above] at (4,13) {2} node[font=\scriptsize,right] {2};
\draw[orange,very thick] (5,13)--(5,8)--(9,8)--(9,4)--(13,4) node[font=\scriptsize,above] at (5,13) {2} node[font=\scriptsize,right] {2};
\draw[orange,very thick] (6,13)--(6,10)--(11,10)--(11,6)--(13,6) node[font=\scriptsize,above] at (6,13) {2} node[font=\scriptsize,right] {2};
\draw[orange,very thick] (10,13)--(10,12)--(12,12)--(12,11)--(13,11) node[font=\scriptsize,above] at (10,13) {2} node[font=\scriptsize,right] {2};

\draw[green!70!black,very thick] (8,13)--(8,7)--(13,7) node[font=\scriptsize,above] at (8,13) {3} node[font=\scriptsize,right] {3};
\draw[green!70!black,very thick] (9,13)--(9,8)--(13,8) node[font=\scriptsize,above] at (9,13) {3} node[font=\scriptsize,right] {3};
\draw[green!70!black,very thick] (11,13)--(11,10)--(13,10) node[font=\scriptsize,above] at (11,13) {3} node[font=\scriptsize,right] {3};
\draw[green!70!black,very thick] (12,13)--(12,12)--(13,12) node[font=\scriptsize,above] at (12,13) {3} node[font=\scriptsize,right] {3};

\foreach \coord in {(7,9), (3,5), (2,3), (1,1)} {
    \node[fill=white] at \coord {$1$};
}
\foreach \coord in {(7,11),(10,9), (3,6),(6,5), (2,4),(5,3), (1,2),(4,1)} {
    \node[fill=white] at \coord {$2$};
}
\foreach \coord in {(1,7),(4,2),(8,1), (2,8),(5,4),(9,3), (3,10),(6,6),(11,5), (7,12),(10,11),(12,9)} {
    \node[customred,draw,circle,fill=white,inner sep=1pt] at \coord {\textbf{3}};
}
\foreach \coord in {(4,7),(8,2), (5,8),(9,4), (6,10),(11,6), (10,12),(12,11)} {
    \node[orange,draw,circle,fill=white,inner sep=1pt] at \coord {$4$};
}
\foreach \coord in {(8,7), (9,8), (11,10), (12,12)} {
    \node[green!70!black,draw,circle,fill=white,inner sep=1pt] at \coord {$5$};
}

\draw[->, line width=1.5pt] (-0.4,-0.4) -- (0.4,0.4);

\node[anchor=west] at (15,6.5) {
$\evacuation(P) = 
\vcenter{\hbox{
\begin{ytableau}
1 & 3 & 5 & 9 \\
2 & 4 & 6 & 11 \\
7 & 8 & 10 & 12
\end{ytableau}
}}$
};

\node[anchor=south] at (7,15) {
$Q = 
\vcenter{\hbox{
\begin{ytableau}
1 & 2 & 3 & 7 \\
4 & 5 & 6 & 10 \\
8 & 9 & 11 & 12
\end{ytableau}
}}$
};
\end{tikzpicture}
\end{center}
We see directly that the successive skeleta of $\mathcal{P}^{\NE}(\prom_3(S))$ are the entries of the promotion permutations $\prom_4(S), \prom_5(S)$ in the northeast block. The other claims in \Cref{thm:main} and \Cref{rem:main-structure} may also be seen directly from this data. For instance, one may imagine casting light from the $3$'s in the $\swarrow$ direction, in which case the skeleton is the location of the $2$'s.
\end{example}

\begin{proof}[Proof of \Cref{thm:prom-RS}]
    Let $\mathcal{P} \coloneqq \mathcal{P}^{\NE}(\rho_r)$. By \eqref{eq:main-1} and \Cref{thm:Viennot-symmetries},
      \[ \mathcal{V}(\mathcal{P}, \nearrow) = (\evacuation(P), Q)\qquad\Rightarrow\qquad\mathcal{V}(\mathcal{P}, \searrow) = (Q^\top, P^\top). \]
    It remains to relate $\mathcal{V}(\mathcal{P}, \searrow)$ and the RS correspondence. Again by \eqref{eq:main-1}, $\mathcal{P}^{\NE}(\rho_r) \in \mathfrak{S}_n$. Hence $\mathcal{P}^{\NE}(\rho_r) = \mathcal{P}(M(\rho_r^{\NE}))$. One may be tempted to directly apply \Cref{thm:Viennot-basic} to incorrectly conclude that $\mathcal{V}(\mathcal{P}, \nearrow) = \RS(\rho_r^{\NE})$, but the convention for translating $\rho_r^{\NE}$ to a subset via $\mathcal{P}^{\NE}(\rho_r)$ differs from that in \Cref{thm:Viennot-basic}. Indeed, comparing the two conventions, we find we must first rotate $\mathcal{P}$ counterclockwise $90^\circ$ to get the subset $\mathcal{P}'$ associated to $\rho_r^{\NE}$ in \Cref{thm:Viennot-basic}. Thus $\mathcal{V}(\mathcal{P}, \searrow) = \mathcal{V}(\mathcal{P}', \nearrow) = \RS(\rho_r^{\NE})$, completing the proof.
\end{proof}

\begin{proof}[Proof of \Cref{cor:prom-RS}]
    Injectivity follows directly from \Cref{thm:prom-RS} since $\RS$ is a bijection. That $\prom_r(S) \in \mathfrak{S}_{2n}$ is a fixed-point free involution is clear from \Cref{sec:background-prom-perms}. The $2$-cycles are as indicated by \Cref{thm:prom-RS} and \Cref{thm:prom-perm-symmetries}(d).

    As for the crossing and nesting numbers, for a perfect matching $M$ on $[2n]$, consider the upper triangular matrix $U$ that has a $1$ in row $i$ and column $j$ when $\{i, j\}$ is a block in $M$ with $i < j$, and $0$'s everywhere else. It is easy to see that the crossing number of $M$ is the length of the longest southeast chain of $1$'s of $U$ that is contained in a rectangle that fully fits into the strictly upper triangular part of $U$. Similarly, the nesting number corresponds to the longest southwest chain of $1$'s that is contained in such a rectangle.

    For a fixed-point free involution $\rho \in \mathfrak{S}_{2n}$ where each $2$-cycle $(a\ b)$ of $\rho$ is of the form $a \in \{1, \ldots, n\}$ and $b \in \{n+1, \ldots, 2n\}$, all $1$'s in the upper triangle of $M(\rho)$ are in $M(\rho)^{\NE}$, which is itself the permutation matrix of $\rho^{\NE}$. From the previous paragraph, the crossing and nesting numbers of $\rho$ are the lengths of the longest southeast and southwest chains of $1$'s, respectively, in $\rho^{\NE}$. By our conventions, a southwest chain of $1$'s corresponds to a decreasing subsequence and a southeast chain of $1$'s corresponds to an increasing subsequence. By Greene's theorem (see, e.g., \cite[Thm.~3.5.3]{Sagan-symmetric-group}), the longest increasing subsequence of $\rho^{\NE}$ is the number of columns in the $\RS$ shape, and the longest decreasing subsequence of $\rho^{\NE}$ is the number of rows in the $\RS$ shape. For a particular $rc=n$, these numbers are $r$ and $c$, respectively, if and only if the $\RS$ shape is $c \times r$. Hence the result follows by \Cref{thm:prom-RS} and the invertibility of the RS correspondence.
\end{proof}

\subsection{Main theorem proof}\label{sec:main-proof}

We now turn to the proof of \Cref{thm:main}. Suppose $S = \stack(P, Q) \in \SYT(2r \times c)$ throughout and let $n \coloneqq rc$.

The northeast $n \times n$ block $\PM(S)^{\NE}$ by definition corresponds to decorations in an $(n+1) \times (n+1)$ block of partitions in $\PEdiagram(S)$. Call these partitions $\lambda(x, y)$ for $0 \leq x, y \leq n$, where we use Cartesian coordinates. As  in \Cref{rem:PEdiagram-pieces}, the $n$-th column of $\PEdiagram(S)$ encodes $\evacuation(P)$, so in our indexing scheme,
  \[ \evacuation(P) = \lambda(0, 0) \to \lambda(0, 1) \to \cdots \to \lambda(0, n). \]

We next give a complete description of $\lambda(x, y)$ in terms of $\lambda(x, 0)$ and $\lambda(0, y)$. First, we introduce the following notion.

\begin{definition}
    Suppose $\lambda, \mu$ are two top-justified collections of boxes (such as integer partitions). The \emph{vertical sum} $\lambda \vsum \mu$ is the collection of top-justified boxes obtained by vertically stacking the columns of $\mu$ onto the columns of $\lambda$.
\end{definition}

For example, if $\lambda$ and $\mu$ are partitions, then $\lambda \vsum \mu = (\lambda' + \mu')'$, where $'$ denotes the conjugate partition and addition is done pointwise. Our arguments will use the additional generality afforded by not necessarily requiring the collections of boxes to be left-justified.

\begin{lemma}\label{lem:proof-vsum}
    We have
    \begin{enumerate}
        \item $\lambda(0,0) \to \lambda(0,1) \to \cdots \to \lambda(0,n) = \evacuation(P)$,
        \item $\lambda(0,0) \to \lambda(1,0) \to \cdots \to \lambda(n,0) = Q$, and
        \item $\lambda(x,y) = \lambda(x,0) \vsum \lambda(0,y)$.
    \end{enumerate}
\end{lemma}

\begin{proof}
    We have already observed (1). For (3), first suppose
      \[ Q = \mu_0 \to \mu_1 \to \cdots \to \mu_n \]
    and let $\tilde{\lambda}(x, y) \coloneqq \mu_x \vsum \lambda(0,y)$. We claim
      \[ \tilde{\lambda}(x,y) = \lambda(x,y). \]
    Along the left edge, we have $\tilde{\lambda}(0, y) = \mu_0 \vsum \lambda(0, y) = \lambda(0, y)$. Along the top edge, we have $\tilde{\lambda}(x, n) = \mu_x \vsum \lambda(0, n) = (r \times c) \vsum \mu_x$, which indeed agrees with $\lambda(x, n)$ since this corresponds to the portion of $\stack(P, Q)$ coming from $Q$. Recall that the $(n+1) \times (n+1)$ block of partitions $\lambda(x, y)$ is uniquely determined by filling in the block with local rules starting from the upper left. Hence to prove the claim it suffices to show that each square $\tilde{\lambda}(x,y), \tilde{\lambda}(x,y+1), \tilde{\lambda}(x+1,y), \tilde{\lambda}(x+1,y+1)$ is a local rule diagram.
    
    For that, first note that each $\tilde{\lambda}(x,y)$ is a partition since the vertical sum of partitions is a partition. Since $\mu_x$ and $\mu_{x+1}$ differ by a single cell, there exists some unique $i$ such that $\mu_{x+1} = \mu_x \vsum c_i$, where we write $c_i$ to denote a single box in the $i$-th column. Likewise there is a unique $j$ such that $\lambda(0,y+1) = \lambda(0,y) \vsum c_j$. Since $\vsum$ is associative and commutative, we then have
    \begin{align*}
        \tilde{\lambda}(x+1,y) &= \tilde{\lambda}(x,y) \vsum c_i \\
        \tilde{\lambda}(x,y+1) &= \tilde{\lambda}(x,y) \vsum c_j \\
        \tilde{\lambda}(x+1,y+1) &= \tilde{\lambda}(x,y) \vsum c_i \vsum c_j.
    \end{align*}
    It only remains to check the sorting condition \eqref{eq:local_rule.sort} for local rule diagrams. Alternatively, we may check the description in \Cref{rem:sorting-cases}. If the two added boxes in $\tilde{\lambda}(x+1,y+1)$ occur in different rows and columns of $\tilde{\lambda}(x, y)$, then clearly $i \neq j$ and we are in case (i) of \Cref{rem:sorting-cases}. If the two added boxes occur in the same row, then $j=i \pm 1$, say $j=i+1$, but then $\tilde{\lambda}(x,y) \vsum c_{i+1}$ is not a partition, so case (ii) of \Cref{rem:sorting-cases} does not occur here. Finally, if the two added boxes occur in the same column, then $j=i$ and we are in case (iii) of \Cref{rem:sorting-cases}, completing the claim.
    
    From the claim, $\lambda(x,0) = \tilde{\lambda}(x,0) = \mu_x \vsum \lambda(0,0) = \mu_x$. Hence we indeed have (2). Now $\tilde{\lambda}(x,y) = \lambda(x,0) \vsum \lambda(0,y)$, so (3) follows from the claim as well.
\end{proof}

It will be convenient to allow Cartesian coordinates for entries of an $n \times n$ matrix $M$, so we write $M(x,y) \coloneqq M_{n+1-y,x}$. For instance, $M(1,1)$ is the lower-left entry of $M$ and $M(1,n)$ is the upper-left entry.

Let
\begin{align*}
    x_{ij} &\coloneqq \text{entry in the $i$-th row, $j$-th column of $Q$} \\
    y_{ij} &\coloneqq \text{entry in the $i$-th row, $j$-th column of $\evacuation(P)$}.
\end{align*}
These entries directly encode $\PM(S)^{\NE}$ as follows.

\begin{lemma}\label{lem:proof-M}
    For all $1 \leq s, t \leq r$ and $1 \leq j \leq c$, we have
    \begin{equation}
        \PM(S)^{\NE}(x_{sj}, y_{tj}) = s+t-1.
    \end{equation}
    Furthermore, this determines all non-zero entries of $\PM(S)^{\NE}$.
\end{lemma}

\begin{proof}
    From the local rule construction in \Cref{sec:background-local-rules}, the entry $\PM(S)^{\NE}(x, y)$ for $1 \leq x, y \leq n$ is the decoration $m$ (if present) in the local rule square $\lambda(x-1, y-1), \lambda(x-1,y), \lambda(x,y-1), \lambda(x,y)$. A decoration is present if and only if sorting occurs, or equivalently if and only if case (iii) of \Cref{rem:sorting-cases} occurs, where the added cells are adjacent in some column $j$ and rows $m, m+1$.
    
    In the course of proving \Cref{lem:proof-vsum}(3), which we now appeal to, we saw that case (iii) occurs here if and only if $\lambda(x, 0) = \lambda(x-1,0) \vsum c_j$ and $\lambda(0,y) = \lambda(0,y-1) \vsum c_j$. By \Cref{lem:proof-vsum}(1)-(2), this says the box labeled $x$ in $Q$ and the box labeled $y$ in $\evacuation(P)$ are each in column $j$, say $x_{sj} = x$ and $y_{tj} = y$. The $j$-th column of $\lambda(x,y) = \lambda(x,0) + \lambda(0,y)$ hence has length $s+t$, so the first added cell is in row $m = s+t-1$. That is,
      \[ \PM(S)(x_{sj}, y_{tj})^{\NE} = \PM(S)^{\NE}(x, y) = m = s+t-1. \]
\end{proof}

From \Cref{lem:proof-M}, the entries in $\PM(S)^{\NE}$ corresponding to $\prom_k(S)$ are those at Cartesian coordinates
  \[ S_k \coloneqq \{(x_{sj},y_{tj}) \mid 1 \leq j \leq c\text{ and }1 \leq s,t \leq r\text{ and }s+t-1=k\}. \]
The case $S_r$ is instructive. Take the same column in both $Q$ and $\evacuation(P)$ and form points from the pairs of entries in complementary rows, giving $r \cdot c = n$ pairs overall.

More precisely, for a fixed $\prom_k(S)$, we may consider pairs coming from the $j$-th columns of $Q$ and $\evacuation(P)$,
  \[ S_{kj} \coloneqq \{(x_{sj},y_{tj}) \mid 1 \leq s, t \leq r\text{ and }s+t-1=k\}. \]
We next show these are in fact the appropriate shadow lines.

\begin{lemma}\label{lem:proof-shadows}
    The shadow lines of $S_k$ are
    \begin{align*}
    \begin{cases}
        \mathcal{L}_j(S_k, \nearrow) = S_{kj}
          & \text{where }1 \leq j \leq c\text{ and }r \leq k \leq 2r-1, \\
        \mathcal{L}_j(S_k, \swarrow) = S_{kj}
          & \text{where }1 \leq j \leq c\text{ and }1 \leq k \leq r.
    \end{cases}
    \end{align*}
    The shadow of $S_k$ is
    \begin{align*}
    \begin{cases}
        \mathcal{S}(S_k, \nearrow) = S_{k+1}
          & \text{where }r \leq k \leq 2r-1, \\
        \mathcal{S}(S_k, \swarrow) = S_{k-1}
          & \text{where }1 \leq k \leq r,
    \end{cases}
    \end{align*}
    where $S_0 = S_{2r} \coloneqq \varnothing$. Moreover,
    \begin{align*}
    \begin{cases}
        \mathcal{S}(S_{kj}, \nearrow) = S_{k+1,j}
          & \text{where }r \leq k \leq 2r-1, \\
        \mathcal{S}(S_{kj}, \swarrow) = S_{k-1,j}
          & \text{where }1 \leq k \leq r.
    \end{cases}
    \end{align*}
\end{lemma}

\begin{proof}
    We take $r \leq k \leq 2r-1$ throughout, the other case being analogous. Set $\ell \coloneq k-r+1 \leq r$. We have the following diagram of relations between the $(r-\ell+1) \cdot c$ entries of $S_k = S_{k1} \sqcup \cdots \sqcup S_{kc}$:
    \[
    \begin{array}{ccccccc}
        S_{k1}
          &
          &
          S_{k2}
          &
          &
          &
          & S_{kc} \\
        \hline
        (x_{\ell1},y_{r1})
          & \nearrow
          & (x_{\ell2},y_{r2})
          & \nearrow
          & \cdots
          & \nearrow
          & (x_{\ell c},y_{rc}) \\
        \searrow
          & \vdots
          & \searrow
          & \cdots
          & \searrow
          & \cdots
          & \searrow \\
        \vdots
          & \nearrow
          & \vdots
          & \nearrow
          & \ddots
          & \nearrow
          & \vdots\\
        \searrow
          & \vdots
          & \searrow
          & \cdots
          & \searrow
          & \cdots
          & \searrow \\
        (x_{r1},y_{\ell 1})
          & \nearrow
          & (x_{r2},y_{\ell 2})
          & \nearrow
          & \cdots
          & \nearrow
          & (x_{rc},y_{\ell c})
    \end{array}
    \]
    Here each $\nearrow$ and $\searrow$ relation follows directly from the row and column increasing conditions of $Q$ and $\evacuation(P)$.

    From these relations, we see directly that $\mathcal{L}_1(S_k, \nearrow) = S_{k1}$, and likewise for later shadows. The first claim follow. For the third, observe
      \[ (x_{\ell j}, y_{r j})
         \nearrow (x_{\ell+1,j}, y_{rj})
         \searrow (x_{\ell+1,j}, y_{r-1,j})
         \nearrow (x_{\ell+2,j}, y_{r-1,j})
         \searrow \cdots. \]
    This sequence interleaves $S_{kj}$ and $S_{k+1,j}$ in the required way. Now the second claim follows from the first and third.
\end{proof}

We may now prove \Cref{thm:main}.

\begin{proof}[Proof of \Cref{thm:main}]
    We begin with \eqref{eq:main-1}. By \Cref{lem:proof-M}, $\mathcal{P}^{\NE}(\rho_r) = S_r$. By \Cref{lem:proof-shadows}, the north rays in the $\ell$-th iteration of Viennot's construction exit the grid in positions $x_{\ell 1}, x_{\ell 2}, \ldots, x_{\ell c}$, and the east rays exit the grid in positions $y_{\ell 1}, y_{\ell 2}, \ldots, y_{\ell c}$. These are the $\ell$-th rows of $Q$ and $\evacuation(P)$, respectively, giving \eqref{eq:main-1}.

    Now consider \eqref{eq:main-2}. By \Cref{lem:proof-M}, $\mathcal{P}^{\NE}(\rho_{r+k}) = S_{r+k}$. Hence \eqref{eq:main-2} follows from \Cref{lem:proof-shadows}. The more precise claim also follows from the final part of \Cref{lem:proof-shadows}.
\end{proof}

\bibliographystyle{amsalphavar}
\bibliography{main}

\end{document}